\documentclass[12pt]{amsart}
\usepackage{amscd,amsmath,amsthm,amssymb}
\usepackage[left]{lineno}
\usepackage{pstcol,pst-plot,pst-3d}
\usepackage{color}
\usepackage{pstricks}
\usepackage{stmaryrd}
\usepackage[utf8]{inputenc}
\usepackage{pstricks-add}
\usepackage{epstopdf}
\usepackage{graphicx}
\usepackage{color}

\newpsstyle{fatline}{linewidth=1.5pt}
\newpsstyle{fyp}{fillstyle=solid,fillcolor=verylight}
\definecolor{verylight}{gray}{0.97}
\definecolor{light}{gray}{0.9}
\definecolor{medium}{gray}{0.85}
\definecolor{dark}{gray}{0.6}

 \def\NZQ{\mathbb}               
 
 \def\QQ{{\NZQ Q}}
 \def\ZZ{{\NZQ Z}}

 \def\FF{{\NZQ F}}
 \def\GG{{\NZQ G}}

 \def\frk{\mathfrak}               

 \def\mm{{\frk m}}

 \def\G{{\mathcal G}}

 \def\B{{\mathcal B}}

 \def\ab{{\mathbf a}}
 \def\bb{{\mathbf b}}
 \def\xb{{\mathbf x}}

 \def\cb{{\mathbf c}}
 
 \def\eb{{\mathbf e}}
\def\ab{{\mathbf a}}

 \def\Soc{{\mathbf Soc}}

 \def\opn#1#2{\def#1{\operatorname{#2}}} 
 \opn\chara{char} \opn\length{\ell} \opn\pd{pd} \opn\rk{rk}
 \opn\projdim{proj\,dim} \opn\injdim{inj\,dim} \opn\rank{rank}
 \opn\depth{depth} \opn\grade{grade} \opn\height{height}
 \opn\embdim{emb\,dim} \opn\codim{codim}
 
 \opn\Tr{Tr} \opn\bigrank{big\,rank}
 \opn\superheight{superheight}\opn\lcm{lcm}
 \opn\trdeg{tr\,deg}
 \opn\reg{reg} \opn\lreg{lreg} \opn\ini{in} \opn\lpd{lpd}
 \opn\size{size} \opn\sdepth{sdepth}
 \opn\link{link}\opn\fdepth{fdepth}\opn\lex{lex}
 \opn\tr{tr}
 \opn\type{type}
 \opn\gap{gap}
 \opn\arithdeg{arith-deg}
 \opn\Deg{Deg}
 \opn\sat{sat}
 \opn\div{div} \opn\Div{Div} \opn\cl{cl} \opn\Cl{Cl}
 \opn\Spec{Spec} \opn\Supp{Supp} \opn\supp{supp} \opn\Sing{Sing}
 \opn\Ass{Ass} \opn\Min{Min}\opn\Mon{Mon}
 \opn\Ann{Ann} \opn\Rad{Rad} \opn\Soc{Soc}
 \opn\Im{Im} \opn\Ker{Ker} \opn\Coker{Coker} \opn\Am{Am}
 \opn\Hom{Hom} \opn\Tor{Tor} \opn\Ext{Ext} \opn\End{End}
 \opn\Aut{Aut} \opn\id{id}
 
 \opn\nat{nat}
 \opn\pff{pf}
 \opn\Pf{Pf} \opn\GL{GL} \opn\SL{SL} \opn\mod{mod} \opn\ord{ord}
 \opn\Gin{Gin} \opn\Hilb{Hilb}\opn\sort{sort}
 \opn\PF{PF}\opn\Ap{Ap}
 \opn\mult{mult}
 \opn\bight{bight}
 \opn\aff{aff}
 \opn\relint{relint} \opn\st{st}
 \opn\lk{lk} \opn\cn{cn} \opn\core{core} \opn\vol{vol}  \opn\inp{inp} \opn\nilpot{nilpot}
 \opn\link{link} \opn\star{star}\opn\lex{lex}\opn\set{set}
 \opn\width{wd}
 \opn\Fr{F}
 \opn\QF{QF}
 \opn\G{G}
 \opn\type{type}\opn\res{res}
 \opn\conv{conv}
 \opn\Shad{Shad}
 \opn\gr{gr}
 
 \def\pot#1#2{#1[\kern-0.28ex[#2]\kern-0.28ex]}

 \opn\dirlim{\underrightarrow{\lim}}
 \opn\inivlim{\underleftarrow{\lim}}
 \let\iso=\cong
 \let\Union=\bigcup
 
 \let\Dirsum=\bigoplus
 
 \let\to=\rightarrow
 
 \def\Implies{\ifmmode\Longrightarrow \else
         \unskip${}\Longrightarrow{}$\ignorespaces\fi}
 \def\implies{\ifmmode\Rightarrow \else
         \unskip${}\Rightarrow{}$\ignorespaces\fi}
 \def\iff{\ifmmode\Longleftrightarrow \else
         \unskip${}\Longleftrightarrow{}$\ignorespaces\fi}

 \let\:=\colon
 \newtheorem{Theorem}{Theorem}[section]
 \newtheorem{Lemma}[Theorem]{Lemma}
 \newtheorem{Corollary}[Theorem]{Corollary}
 
 \newtheorem{Remark}[Theorem]{Remark}
 \newtheorem{Remarks}[Theorem]{Remarks}

 \newtheorem{Definition}[Theorem]{Definition}

 \let\epsilon\varepsilon
 \let\kappa=\varkappa
 \def\qed{\ifhmode\textqed\fi
       \ifmmode\ifinner\quad\qedsymbol\else\dispqed\fi\fi}
 \def\textqed{\unskip\nobreak\penalty50
        \hskip2em\hbox{}\nobreak\hfil\qedsymbol
        \parfillskip=0pt \finalhyphendemerits=0}
 \def\dispqed{\rlap{\qquad\qedsymbol}}
 
 \opn\dis{dis}
 \def\pnt{{\raise0.5mm\hbox{\large\bf.}}}
 
 \opn\Lex{Lex}

\begin{document}

\title {The saturation number of $\cb$-bounded stable  monomial ideals and their powers }

\author {Reza Abdolmaleki, J\"urgen Herzog and Guangjun Zhu$^{^*}$}

\address{Reza Abdolmaleki: Department of Mathematics, Institute for Advanced Studies in
Basic Sciences (IASBS), 45195-1159 Zanjan, Iran}
\email{abdolmaleki@iasbs.ac.ir}

\address{J\"urgen Herzog: Fachbereich Mathematik, Universit\"at Duisburg-Essen, Campus Essen, 45117
Essen, Germany} \email{juergen.herzog@uni-essen.de}

\address{Guangjun Zhu: School of Mathematical Sciences, Soochow University,
 Suzhou 215006, P. R. China}\email{zhuguangjun@suda.edu.cn(Corresponding author: Guangjun Zhu)}
 \thanks{* Corresponding author}

\thanks{This research is supported by the National Natural Science Foundation of China (No.11271275) and by foundation of the Priority Academic Program Development of Jiangsu Higher Education Institutions.
This paper was written  while the third author was visiting the Department of Mathematics of University Duisburg-Essen, Germany. She  spent a memorable time at Essen, so she would like to express her  hearty thanks to Maja for hospitality. }

\address{}
\email{}

\address{}
\email{}

\dedicatory{ }

\begin{abstract} Let $S=K[x_1,\ldots,x_n]$ be the polynomial ring in $n$ variables over a
field $K$.
In this paper, we compute  the socle of $\cb$-bounded strongly stable ideals and determine that the saturation number of  strongly stable ideals and of equigenerated  $\cb$-bounded strongly stable ideals.
We also  provide   explicit formulas for the saturation number $\sat(I)$  of  Veronese type ideals $I$. Using this formula, we show that $\sat(I^k)$ is quasi-linear from the beginning and we determine the quasi-linear function explicitly.
\end{abstract}




\maketitle

\setcounter{tocdepth}{1}


\section*{Introduction}
In recent years there has been a lot of work on algebraic and homological properties of powers of graded  ideals in the polynomial ring $S=K[x_1,\ldots,x_n]$, where $K$ is a field. Typically, many of the invariants known  behave asymptotically well, that is, stabilize  or show a regular behaviour  for sufficiently high powers of $I$. Classical examples of this feature are Brodmann's results  \cite{B1} and \cite{B2} which say that $\depth S/I^k$ is constant for $k\gg 0$ and $\Ass(I^{k+1})=\Ass(I^k)$ for $k\gg 0$, or the result by Cutkosky, Herzog, Trung \cite{CHT} and Kodyalam \cite{K} which says that the regularity of $I^k$ is a linear function for $k\gg 0$.

Recently it was noted in \cite{HKM} that for $k\gg 0$, $\sat(I^k)$ is a quasi-linear function provided $I$ is a monomial ideal. Here, $\sat(I)$ denotes the saturation number of a graded ideal $I\subset S$, that is, the smallest number $\ell$ for which $I:\mm^{\ell+1}=I:\mm^\ell$, where $\mm=(x_1,\ldots, x_n)$ is the unique graded maximal ideal of $S$. Such number exists because $S$ is Noetherian and $I\subseteq I:\mm\subset I:\mm^2\subseteq\ldots $. The ideal $I^{\sat}=\Union_{\ell\geq 0}(I:\mm^\ell)$ is called the saturation of $I$. Thus $\sat(I)$  tells us how many steps are needed to reach $I^{\sat}$.

If $I\subset S$ is a strongly stable ideal, then $\sat(I)=\max\{\ell\:\; \text{ $x_n^\ell|u$  for $u\in G(I)$}\}$, see Theorem~\ref{proffind}. Here, $G(I)$ denotes the unique minimal set of monomial generators of $I$. From this result one easily deduces (Corolary~\ref{plus})  that for two strongly stable ideals $I$ and $J$,  one  has $\sat(IJ)\leq \sat(I)+\sat(J)$ with equality if $I$ and $J$ are equigenerated. If either $I$ or $J$ is not equigenerated, then this inequality may be strict, and it fails to be true if the ideals $I$ and $J$ are not strongly stable. For example, if we consider the ideal $I=(x_1x_2,x_1x_3,x_2x_3)$,  then $\sat(I)=0$ and $\sat(I^2)=1$. Of course, $I$ is not strongly stable, but it is squarefree strongly stable.  More generally we may consider $\cb$-bounded strongly stable ideals, where $\cb\in \ZZ^n$ is an integer vector. We call $I$ to be $\cb$-bounded strongly stable, if $I$ is a monomial ideal, and (i) for all $u=x_1^{a_1}\cdots x_n^{a_n}\in G(I)$ we have $a_i\leq c_i$,  and (ii) whenever  $u\in G(I)$ and  $i<j$  with $x_j|u$  and  $x_iu/x_{j}$ is $\cb$-bounded, it follows that $x_iu/x_{j}\in I$.

In the first section we consider the socle of $\cb$-bounded strongly stable ideals and prove in Theorem~\ref{d-1} that if $I$ is such an ideal and is generated in degree $d$, then $I:\mm=I+J$, where $J$ is generated in degree $d-1$ and  is $(\cb-\eb)$-bounded strongly stable.  Here, $\eb=(1,1,\ldots,1)$.

In Section~2  we determine that the saturation number of equigenerated  $\cb$-bounded strongly stable ideals and prove on Theorem~\ref{sat} for such an ideal $I$, the saturation number of $I$ is the  maximal number $\ell$  for which  there exists  $u\in G(I)$  such that  $x_n^\ell|u$ and   the multidegree of $x_n^\ell|u$ is componentwise bounded above by $\cb-\ell\eb$. Examples show that this formula for $\sat(I)$ may fail, when $I$ is not equigenerated or $I$ is only a stable ideal. In Section~3 we  apply  the formula for $\sat(I)$  given in Theorem~\ref{sat}  to determine the function $f(k)=\sat(I^k)$ when $I$ is a $\cb$-bounded principal strongly stable ideal, see Corollary~\ref{borelpower}. For the proof we need a fact, shown in Theorem~\ref{soclepower} that the $k$th power  of $\cb$-bounded principal strongly stable ideal is a  $k\cb$-bounded principal strongly stable ideal. This may fail, if $I$ is an equigenerated strongly stable but not principal strongly stable and it also may fail if $I$ is principal  stable but not strongly stable.

In the last section we make we give a more explicit formula for $\sat(I)$ when $I$ is an ideal of Veronese type. Given a positive integers  $n$, an integer $d$ and   an  integer vector $\ab=(a_1,\ldots,a_n)$ with
 $a_1\geq a_2\geq \cdots\geq a_n$, one defines the
monomial ideal  $I_{\ab,n,d}\subset S=K[x_1,\ldots,x_n]$ with
\[
G(I_{\ab,n,d})=\{x_1^{b_1}x_2^{b_2}\cdots x_n^{b_n}\; \mid \; \sum_{i=1}^nb_i=d \text{ and  $b_i\leq a_i$ for $i=1,\ldots,n$}\}.
\]
Ideals of this type are called of {\em Veronese type}. It is obvious that $I_{\ab,n,d}$ is $\cb$-bounded strongly stable.  The converse is not always true. In Theorem~\ref{satveronese} it is shown that if  $I_{\ab,n,d}$ is  a  Veronese type ideal with $n>1$, $d\geq 0$,  $a_1\geq a_2\geq \cdots \geq a_n\geq 0$ and $\sum_{i=1}^na_i\geq d$. Then  $\sat(I_{\ab,n,d})=\min\bigg\{\bigg\lfloor \frac{\sum\limits_{i=1}^{n}a_i-d}{n-1}\bigg\rfloor, a_n, d\bigg\}$,
where $\lfloor  a\rfloor$ is the largest integer less than or equal  to $a$.

For any monomial ideal, the function $f(k)=\sat(I^k)$ is quasi-linear for $k\gg 0$, as noticed in \cite{HKM}. We use the formula for the saturation number of   a Veronese type ideal to show in Theorem~\ref{qusilear} that for Veronese type ideals, $\sat(I^k)$ is quasi-linear from the very beginning and we determine the quasi-linear function explicitly.

\section{The socle of  $\cb$-bounded stable ideals}

Let $K$ be a field and  $S=K[x_1,\ldots,x_n]$  be the polynomial ring over $K$ in the variables $x_1,\ldots,x_n$. The set of monomials of $S$ will be denoted by $\Mon(S)$. Let $u\in \Mon(S)$, then $u=x_1^{a_1}\cdots x_n^{a_n}$ and we write $u=\xb^\ab$ where $\ab=(a_1,\ldots,a_n)$.  The multidegree of $u$ is defined to be $\Deg(u)=\ab$. We also  set $m(u)=\max\{i\:\; a_i\neq 0\}$.
An ideal $I\subset S$ is called a {\em monomial ideal} if it is generated by monomials.
The unique minimal set of monomial generators
of $I$ will be denoted by $G(I)$.

Let $\cb=(c_1,\ldots,c_n)$ be an integer vector with $c_i\geq 0$. The monomial  $u=x_1^{a_1}\cdots x_n^{a_n}$  is called {\em $\cb$-bounded},
 $\ab\leq \cb$, that is, $a_i\leq c_i$  for all $i$.
Let $I$ be a monomial ideal generated by the monomials $u_1,\ldots, u_m$.  We set
\[
I^{\leq \cb}=(u_i\:\; \text{ $u_i$  is  $\cb$-bounded}).
\]

\begin{Definition}
{\em Let $I\subset S$ be a  $\cb$-bounded monomial ideal.}
\begin{enumerate}
{\em \item[(a)] $I$ is called {\em $\cb$-bounded  stable} if for all $u\in G(I)$ and all $i<m(u)$ for which   $x_iu/x_{m(u)}$ is $\cb$-bounded, it follows that $x_iu/x_{m(u)}\in I$.}
{\em \item[(b)]  $I$ is called {\em $\cb$-bounded  strongly stable} if for all  $u\in G(I)$ and all $i<j$  with $x_j|u$  and    $x_iu/x_{j}$ is $\cb$-bounded, it follows that $x_iu/x_{j}\in I$.}
\end{enumerate}
\end{Definition}

Let $u_1,\ldots,u_m\in \Mon(S)$ be $\cb$-bounded. The smallest $\cb$-bounded strongly stable ideal containing $u_1,\ldots,u_m$ is denoted by $B^\cb(u_1,\ldots,u_m)$. A monomial ideal $I$ is called a $\cb$-bounded  strongly  stable principal ideal, if there exists a $\cb$-bounded monomial $u$ such that $I=B^{\cb}(u)$.  The smallest strongly stable ideal containing $u_1,\ldots,u_m$ (with no restrictions on the exponents) is denoted $B(u_1,\ldots,u_m)$. The monomials $u_1,\ldots,u_m$ are called {\em Borel generators}  of $I=B(u_1,\ldots,u_m)$.

Similar definitions can be made for stable ideals. The $\cb$-bounded smallest stable ideal containing $u_1,\ldots, u_m$ will be denoted by $\B^\cb(u_1,\ldots,u_m)$, and the  elements $u_1,\ldots, u_m$ are called {\em stable Borel generators} of $I=\B^\cb(u_1,\ldots,u_m)$.

\medskip
Let $I\subset S$ be a graded ideal.   We have   the following ascending chain of ideals $I\subseteq I:\mm\subseteq I:\mm^2\subseteq \ldots$. Since $S$ is Noetherian, there exists an integer $k\geq 0$ such  that $I:\mm^k=I:\mm^{k+1}$. We set
\[
\sat(I)=\min\{k\:\; I:\mm^k=I:\mm^{k+1}\}.
\]

\medskip
We start with the following result.

\begin{Theorem}
\label{proffind}
Let $I$ be a strongly stable ideal.  Then
\[
\sat(I)=\max\{\ell:\; x_n^{\ell}|u\ \ \text{for some }u\in G(I)\}.
\]
\end{Theorem}
\begin{proof} Let $s=\max\{\ell:\; x_n^{\ell}|u \ \text{for some } u\in G(I)\}$. Let $G(I)=\{u_1,\ldots, u_m\}$.
We prove the statement by induction on $s$. We use repeatedly the fact that $I:\mm=I:x_n$ and that $I:\mm$ is strongly stable  is strongly stable, because  $I$ is  strongly stable.

If $s=0$, then  $x_n\nmid u_i$ for $i=1,\ldots,m$. It follows that $I:\mm=I:x_n=I$. Hence $\sat(I)=0$.
Now we assume that $s\geq 1$.  Furthermore, we may assume that $x_n\nmid u_\ell$ for $\ell=1,\ldots,i$, and
$x_n|u_\ell$ for $\ell=i+1,\ldots,m$. Then $G(I:\mm)=\{u_{i_1},\ldots, u_{i_t},u_{i+1}/x_n,\ldots, u_m/x_n\}$,   where $\{u_{i_1},\ldots, u_{i_t}\}$ is a suitable subset of $\{u_1,\ldots,u_i\}$.  Indeed, $\{u_1,\ldots,u_i, u_{i+1}/x_n.\ldots u_m/x_n\}$ is a set  of generators of $I:\mm$. Suppose $u_r|u_t/x_n$ for some $1\leq r\leq i$ and $i+1\leq t\leq m$. Then $x_nu_r$ divides $u_t$, a contradiction. It is also clear that $u_r/x_n$ and $u_t/x_n$  can not divide each other, unless $r=t$. This shows that the monomials $u_{i+1}/x_n,\ldots, u_m/x_n$ belong to $G(I:\mm)$, and this yields the assertion.

It follows that $\max\{\ell:\; x_n^\ell|v  \ \text{for some } v\in G(I:\mm)\}=s-1$.
By our induction  hypothesis, we have $\sat(I:\mm)=s-1$. Hence $\sat(I)=\sat(I:\mm)+1=s$.
\end{proof}

\begin{Corollary}
\label{plus}
Let $I$ and $J$ be two strongly stable ideals, then $IJ$ is a strongly stable ideal and
$\sat(IJ)\leq \sat(I)+\sat(J)$.
If $I$ and  $J$ are equigenerated, then $\sat(IJ)=\sat(I)+\sat(J)$.
\end{Corollary}
\begin{proof} Let $w\in IJ$ and $x_j|w$. We may write $w=uv$ with $u\in I$, $v\in J$  and may assume $x_j|u$.
For any  $i<j$, we get $x_iu/x_j\in I$ since $I$ is strongly stable. It follows that $x_iw/x_j=(x_iu/x_j)v\in IJ$. Hence $IJ$ is  strongly stable ideal.

Let $G(I)=\{u_1,\dots,u_r\}$, $G(J)=\{v_1,\ldots,v_s\}$ and $G(IJ)=\{w_1,\ldots,w_t\}$, Then
$\{w_1,\ldots,w_t\}\subseteq \{u_1v_1,\ldots,u_1v_s, u_2v_1,\ldots,u_rv_s\}$. It follows from Theorem~\ref{proffind} that
\[
\sat(IJ)\leq \sat(I)+\sat(J).
\]
If $I$ and  $J$ are equigenerated, then $\{w_1,\ldots,w_t\}=\{u_1v_1,\ldots,u_1v_s, u_2v_1,\ldots,u_rv_s\}$. Thus $\sat(IJ)=\sat(I)+\sat(J)$ from Theorem~\ref{proffind}.
\end{proof}

\begin{Remark}
\label{funny}
{\em
(a) We may have  $\sat(IJ)< \sat(I)+\sat(J)$    if $I$ and $J$ are strongly stable but either $I$ or $J$  is not equigenerated.
For example, let $I=B(x_2^2x_3^2, x_1x_3)$, $J=B(x_1x_3^2, x_2^2x_3)$.  Then $\sat(I)=\sat(J)=2$ and   $\sat(IJ)=3$.

(b) If $I$ and $J$ are not strongly stable, then none of the inequalities of Corollary~\ref{plus} may be valid. For example, if $I=(x_1x_2,x_1x_3,x_2x_3)$, then $1=\sat(I^2)>2\sat(I)=0$.
}
\end{Remark}

 Note that the ideal in the example of  Remark~\ref{funny}(b) is a principal squarefree strongly stable, but fails the inequality given Corollary~\ref{plus} even for powers. Observe  that squarefree monomial ideals are $(1,1,\ldots,1)$-bounded.
Therefore, for the rest of the paper,  we try at least to understand the behaviour the function $f(k)=\sat (I^k)$, when $I$
is a $\cb$-bounded  strongly stable principal ideal.

\medskip
Let  $\eb=(1,1,\ldots,1)\in \ZZ^n$. Then we have
\begin{Theorem}
\label{d-1}
Let $I$ be a non-zero  $\cb$-bounded stable ideal generated in degree $d$. Then $I:\mm=I+J$,  where $J$ is a $(\cb-\eb)$-bounded  ideal generated in degree $d-1$. Indeed,
\[
J=(u/x_n\:\; u\in G(I),\; x_n|u \text{ and }  \Deg (u/x_n)\leq \cb-\eb).
\]
Moreover, if $I$ be a $\cb$-bounded strongly stable ideal,  then $J$ is a $(\cb-\eb)$-bounded strongly stable ideal.
\end{Theorem}

For the proof of the theorem we need the following

\begin{Lemma}
\label{socleres}
Let $I\subset S$ be a   monomial ideal with minimal   multigraded free $S$-resolution
\[
0\to F_{n-1}\to \cdots \to F_1\to F_0\to I\to 0.
\]
Suppose that $F_{n-1}= \Dirsum\limits_{i=1}^{r} S(-\ab_i)$. Then the elements $\xb^{\ab_i}/x_1\cdots x_n$ ($i=1,\ldots,r$) are monomials in $S$ and
\[
(\xb^{\ab_1}/x_1x_2\cdots x_n)+I, \ldots, (\xb^{\ab_r}/x_1x_2\cdots x_n)+I
\]
is  a $K$-basis of $(I:\mm)/I$.
\end{Lemma}

\begin{proof}
 There exists the following isomorphisms of graded modules.
\[
\Dirsum_{i=1}^rK(-\ab_i)\iso \Tor_{n-1}(K, I)\iso \Tor_n(K, S/I)\iso H_n(x_1,\ldots,x_n;S/I).
\]
Here $H_n(x_1,\ldots,x_n;S/I)$ denotes the $n$th Koszul homology of $S/I$ with respect to the sequence $x_1,\ldots,x_n$.
Note that $H_n(x_1,\ldots,x_n;S/I)=(I:\mm/I)\bigwedge^n E$, where $E=\Dirsum_{i=1}^n(S/I)e_i$,   and hence $(I:\mm/I)\bigwedge^n E=(I:\mm/I)e_1\wedge e_2\wedge \cdots \wedge e_n$.

Therefore,  for each $i$  there exists $z_i:=(u_i+I)e_1\wedge e_2\wedge \cdots \wedge e_n\in H_n(x_1,\ldots,x_n;S/I)$, where $u_i\in I:\mm$ is   a monomial,  and $z_i$ has  multidegree $\ab_i$. Moreover, $z_1,\ldots,z_r$  is a $K$-basis  of $H_n(x_1,\ldots,x_n;S/I)$. This implies that $u_1+I, \ldots,u_r+I$ is a $K$-basis of $I:\mm/I$. Comparing multidegrees we see that $u_i=\xb^{\ab_i}/x_1\cdots x_n$ for $i=1,\ldots,r$.
\end{proof}

\begin{proof}[Proof of Theorem~\ref{d-1}] Since $I$ is a non-zero $\cb$-bounded stable ideal, there exist $\cb$-bounded monomials $u_1,\ldots,u_m\in I$ of degree $d$ such that $I=\B^{\cb}(u_1,\ldots,u_m)$. By \cite[Lemma~\ref{principalbound}]{HMRZ} we have   $\B^{\cb}(u_i)  =  \B(u_i)^{\leq \cb}$  for all $i$. Hence
\begin{eqnarray}
\label{cb}
I&=&\B^{\cb}(u_1,\ldots,u_m)=\B^{\cb}(u_1)+\cdots +\B^{\cb}(u_m)=\B(u_1)^{\leq \cb}+\cdots +\B(u_m)^{\leq \cb}\\
&=&\B(u_1,\ldots,u_m)^{\leq \cb}.\nonumber
\end{eqnarray}

Let $\GG$ be the  minimal  multigraded free $S$-resolution of $\B(u_1,\ldots,u_m)$. By the theorem of Eliahou-Kervaire \cite{EK} it follows that  $G_{n-1}= \Dirsum_i S(-\ab_i)$  where for each $\ab_i$  the monomial $\xb^{\ab_i}$ is of the form $x_1\cdots x_{n-1}u$  with $u\in G(I)$ and  $m(u)=n$. Let $\FF$ be the  minimal multigraded free $S$-resolution of $I$. From  (\ref{cb}) and the Restriction Lemma (\cite[Lemma 4.4]{HHZ}) it follows that
$F_{n-1}= \Dirsum_i S(-\ab_i)$  where for  each $\ab_i$  the monomial $\xb^{\ab_i}$ is of the form $x_1\cdots x_{n-1}u$  with $u\in G(I)$,   $m(u)=n$
and $\Deg(x_1\cdots x_{n-1}u)\leq \cb$.

Lemma~\ref{socleres} implies that the elements $u/x_n$ with $u\in G(I)$,  $x_n|u$ and $\Deg(u/x_n)\leq\cb-\eb$ are the generators of $J$.

Now assume that $I$ is a $\cb$-bounded strongly stable ideal in degree $d$. Let $w\in G(J)$ and assume that $x_j|w$  and $w'=x_i(w/x_j)$ is $(\cb-\eb)$-bounded. Then $v=wx_n\in G(I)$, and since  $I$  is  $\cb$-bounded strongly stable and  $v'=x_i(v/x_j)$ is $\cb$-bounded, it follows that $v'\in G(I)$. This  implies that $w'=v'/x_n\in J$.
\end{proof}

\begin{Remark}
\label{reza}{\em
(a) The second part  of Theorem~\ref{d-1} is not satisfied for $\cb$-bounded stable ideals. For example, the ideal $I=(x_1^3, x_1^2x_2, x_1x_2^2, x_1x_2x_3) \subset K[x_1,x_2,x_3]$ is  a $\cb$-bounded stable ideal of degree $3$, where $\cb=(3,2,1)$, and  $J=(x_1x_2)$. The ideal $J$ is not  $(\cb-\eb)$-bounded stable, because $x_1^2\not\in J$.

(b) The second part  of Theorem~\ref{d-1} is not satisfied if $I$  is not   equigenerated, even if it is $\cb$-bounded strongly stable. Indeed,  $I=(x_1^3,x_1^2x_2^2,x_1^2x_2x_3,x_1^2x_3^2)$ be a $\cb$-bounded strongly stable ideal where $\cb=(3,2,2)$. But  $J=(x_1^2x_2,x_1^2x_3)$
 is not  $(\cb-\eb)$-bounded  strongly stable ideal.}
\end{Remark}

\medskip

\section{The saturation number for $\cb$-bounded strongly stable ideals}

Let $I\subset S$ be a graded ideal. Then we define the saturation of $I$ to be the ideal $$I^{\sat}=\Union_k (I:\mm^k).$$
If $I\subset S$ is a monomial ideal. For each $\ell\geq 1$, the $K$-vector space $(I:\mm^{\ell})/(I:\mm^{\ell-1})$ has unique $K$-basis of the form $u_1+(I:\mm^{\ell-1}),u_2+(I:\mm^{\ell-1}),\ldots, u_r+(I:\mm^{\ell-1})$, where the $u_i$ are monomials. We set
\[
J_0(I)=I \text{ and } J_\ell(I)=(u_1,\ldots,u_r) \text{ if }\ell\geq 1.
\]

\begin{Lemma}
\label{d-1already}
Let $I$ be a monomial ideal with $d$-linear resolution. Then $J_1(I)$ is generated in degree $d-1$.
\end{Lemma}

\begin{proof}
Let $\FF$ be the multigraded minimal free resolution of $I$, and let $F_{n-1}=\Dirsum_{i=1}^rS(-\ab_i)$. By Lemma~\ref{socleres},  $J_1(I)$ is generated by the monomials  $\xb^{\ab_i}/x_1\cdots x_n$. Since $I$ has $d$-linear resolution, it follows that $\deg \xb^{\ab_i}=d+n-1$ for all $i$. Therefore,  $J_1(I)$ is generated  by monomials of degree  $d-1$.
\end{proof}

\begin{Lemma}
\label{jk}
With the assumptions and notation introduced we have
\begin{enumerate}
\item[(a)] $I:\mm^\ell =\sum_{k=0}^\ell J_k(I)$.
\item[(b)]$\sat(I)=\max\{\ell\:\; J_\ell(I)\neq 0\}$.
\item[(c)] $I^{\sat} =\sum_{\ell\geq 0} J_\ell(I)$.
\item[(d)] Let $\ell\geq 0$. Suppose that $J_i(I)$ has a $(d-i)$-linear resolution for $i=0,\ldots,\ell$. Then $J_\ell(I):\mm=J_\ell(I)+J_{\ell+1}(I)$. In particular, $J_1(J_\ell(I))=J_{\ell+1}(I)$.
\end{enumerate}
\end{Lemma}
\begin{proof} (a), (b) and (c) are obvious.

Proof of (d): We may assume that $J_\ell(I)\neq 0$, otherwise the assertion is trivial. We have $I:\mm^\ell +J_{\ell+1}(I)=I:\mm^{l+1}$. Therefore, $\mm J_{\ell+1}(I) \subset \sum_{k=0}^\ell J_k(I)$. It follows that the generators of $J_{\ell+1}(I)$ have degree $\geq d-\ell-1$, since by (a) and our assumption  the least degree of generators of $I:\mm^\ell$ is $d-\ell$. Assume $J_{\ell+1}(I)$ has a monomial generator $u$ with $\deg(u)\geq d-\ell$. Then $x_i u\in \sum_{k=0}^{\ell-1}J_k(I)=I:\mm^{\ell-1}$ for $1\leq i\leq n$. Therefore, $u\in I:\mm^{\ell}$, a contradiction. This shows all generators of $J_{\ell+1}(I)$  are of degree $d-\ell-1$. Now let $u\in G(J_{\ell+1}(I))$. Then  $x_iu\in J_{\ell+1}(I)$ for all $i$.     This implies $J_\ell(I):\mm\supseteq J_\ell(I)+J_{\ell+1}(I)$.

Let  $u\in G(J_\ell(I):\mm)$. Then $\deg u=d-\ell-1$, and hence $u\not\in I:\mm^{\ell}$. Thus  $u\in J_{\ell+1}(I)$, and this implies that $J_\ell(I):\mm\subseteq J_\ell(I)+J_{\ell+1}(I)$.

Notice that  $J_\ell(I)+J_{\ell+1}(I)=J_\ell(I):\mm=J_\ell(I)+J_1(J_{\ell}(I))$ and  $J_{\ell+1}(I)$ is generated in degree $d-\ell-1$ from the above proof. Since $J_\ell(I)$ has a $(d-\ell)$-linear resolution, we get $J_1(J_\ell(I))$ is generated in degree $d-\ell-1$. It follows that
$J_1(J_\ell(I))=J_{\ell+1}(I)$.
\end{proof}

\medskip
Now, we prove the main results of this section.

\begin{Theorem}
\label{sat}
Let $I$ be an  equigenerated $\cb$-bounded strongly stable ideal. Then
\begin{enumerate}
\item[(a)] for all $\ell\geq 1$,  $J_{\ell}(I)$ is a $(\cb-\ell\eb)$-bounded strongly stable ideal generated in degree $d-\ell$,
and
\[
J_{\ell}(I)= (u/x_n^{\ell}\:\; u\in G(I),\; x_n^\ell|u \text{ and }  \Deg (u/x_n^{\ell})\leq \cb-\ell\eb) \text { if } J_{\ell-1}(I)\neq 0.
\]
\item[(b)] $\sat(I)$ is the maximal number $\ell$  for which  there exists  $u\in G(I)$  such that  $x_n^\ell|u$ and   $\Deg(u/x_n^\ell)\leq\cb-\ell\eb$.
\end{enumerate}
\end{Theorem}

\begin{proof}
We prove (a) by induction on $\ell$. For $\ell=1$, the assertion from  Therem~\ref{d-1}.   Now let $\ell\geq 2$,  and assume that (a) holds for $\ell-1$. Since by induction hypothesis $J_{\ell-1}(I)$  is $(\cb-(\ell-1)\eb)$-bounded strongly stable, again using Theorem \ref{d-1}, we obtain  $J_{\ell}(I)$ is a $(\cb-\ell\eb)$-bounded strongly stable ideal generated in degree $d-\ell$ and
\[
J_{\ell}(I)= \{v/x_n\:\; v\in G(J_{\ell-1}(I)),\;  x_n|v \text{ and }  \Deg (v/x_n)\leq \cb-\ell\eb\}.
\]
The induction hypothesis implies that
\[
G(J_{\ell-1}(I))= \{u/x_n^{\ell-1}\:\; u\in G(I),\; x_n^{\ell-1}|u \text{ and }  \Deg (u/x_n^{\ell-1})\leq \cb-(\ell-1)\eb\}.
\]
It follows that $v$ is of the form $u/x_n^{\ell-1}$,  where $x_n^{\ell-1}|u$ and $\Deg (u/x_n^{\ell-1})\leq \cb-(\ell-1)\eb$. Hence $v/x_n$ has the form $u/x_n^{\ell}$,  where $x_n^{\ell}|u$ and $\Deg (u/x_n^{\ell})\leq \cb-\ell\eb$, as desired.

(b) Let $s=\sat(I)$ and  $k$ be the maximal number  $\ell$ with the properties described in part (b) of the theorem.
Then $J_s(I)\neq 0$,  by Lemma \ref{jk} (b). This implies that $J_{s-1}(I)\neq 0$. By (a), we get $J_{s}(I)= (u/x_n^{s}\:\; u\in G(I),\; x_n^s|u \text{ and }  \Deg (u/x_n^{s})\leq \cb-s\eb)$. It follows that $s\leq k$.  Suppose that  $s<k$. Then $J_{s+1}(I)\neq 0$,  This contradicts Lemma~\ref{jk}(b).
\end{proof}

\begin{Remark}{\em
(a) Part (b)   of Theorem~\ref{sat} does not hold if $I$ is not equigenerated.
For example, let $I=(x_1,x_2^4,x_2^3x_3,x_2^2x_3^2)\subset K[x_1,x_2,x_3]$, then $I$ is $\cb$-bounded strongly stable where $\cb=(1,4,2)$. By CoCoA, we get $sat(I)=2$. But
$\max\{k\: x_n^k|u \text{ and  } \Deg(u/x_n^k)\leq\cb-k\eb \text{ for $u\in G(I)$}\}=1$.

(b) By the example in Remark~\ref{reza}(a),  $J_1(I)$ need  not to be stable if $I$ is an equigenerated stable ideal. Therefore, we can not apply an induction argument as used in Theorem~\ref{sat}(b).   Nevertheless,  Theorem~\ref{sat}(b) may be valid for any stable equigenerated monomial ideal,  as many explicit examples indicate.}
\end{Remark}

\medskip

\section{The saturation number of powers of $\cb$-bounded strongly stable monomial ideals}

Let $u, v$ be $\cb$-bounded monomials of same degree $d$.   Then we write $v\prec_\cb u$ if and only if  $v\in B^\cb(u)$. This is a partial order on the $\cb$-bounded monomials of degree $d$. We also write $v\prec u$ if and only if $v\in B(u)$.

\begin{Theorem}
\label{soclepower}
Let $u=x_{i_1}\cdots x_{i_d}$ be a $\cb$-bounded monomial in $S$ with $i_1\leq i_2\leq \cdots \leq i_d$ and $I=B^{\cb}(u)$. Then for any positive integer $k$
\begin{itemize}
\item[(a)] $I^k=B^{k\cb}(u^k)$;
\item[(b)] $I^k:\mm=I^k+B^{k\cb-\eb}(u^k/x_n)$, if $i_d=n$, otherwise $I^k:\mm=I^k$;
\item[(c)]for all $\ell\geq 0$ such that $x_n^\ell|u^k$,  $J_{\ell}(I^k)=B^{k\cb-\ell\eb}(u^k/x_n^\ell)$.
\end{itemize}
\end{Theorem}
\begin{proof} Let $u^k=x_{j_1}x_{j_2}\cdots x_{j_{kd}}$ with $j_1\leq j_2\leq \cdots \leq j_{kd}$. Then
$j_{tk+1}=j_{tk+2}=\cdots=j_{tk+k}=i_{t+1}$ for $t=0,1,\ldots, d-1$.

(a) The inclusion $I^k\subseteq B^{k\cb}(u^k)$ is obvious. Conversely, let $w=x_{\ell_1}x_{\ell_2}\cdots x_{\ell_{kd}}$ $\in B^{k\cb}(u^k)$
with $\ell_1\leq \ell_2\leq \cdots \leq \ell_{kd}$, then $\ell_{s}\leq j_{s}$ for any $s=1,\ldots,kd$.
Choose $v_1=x_{\ell_1}x_{\ell_k+1}\cdots x_{\ell_{k(d-1)+1}}$, $v_2=x_{\ell_2}x_{\ell_k+2}\cdots x_{\ell_{k(d-1)+2}}$, $\ldots$,
$v_k=x_{\ell_k}x_{\ell_k+k}\cdots x_{\ell_{kd}}$, then $v_i\in B(u)$ for $i=1,\ldots,k$. Since $w$ is $k\cb$-bounded, we get
each $v_i$ is $\cb$-bounded. This implies that $w\in I^k$.

(b) If $i_d<n$, then it is clear $I^k:\mm=I^k$. Now we assume that $i_d=n$.
Since $I^k=B^{k\cb}(u^k)$ is a $k\cb$-bounded strongly stable ideal and since $u^k/x_n\in I^k:\mm$, it follows from Theorem~\ref{d-1} that  $B^{k\cb-\eb}(u^k/x_n)\subseteq I^k:\mm$. Hence
 $I^k+B^{k\cb-\eb}(u^k/x_n)\subseteq I^k:\mm$.  Conversely,
$v\in G(I^k:\mm)\setminus I^k$, then $\mm v\subset I^k$, where   $\deg(v)=kd-1$ and $\Deg(v)\leq k\cb-\eb$
by Theorem \ref{d-1}. Let $v=x_{s_1}\cdots x_{s_{kd-1}}$ with $s_1\leq s_2\leq \cdots \leq s_{kd-1}\leq n$, then $x_nv=x_{s_1}\cdots x_{s_{kd-1}}x_n\in B^{k\cb}(u^k)$ by part (a). It follows that $s_\ell\leq j_\ell$ for $\ell=1,\ldots,kd-1$.
 This means that $v\in B^{k\cb-\eb}(u^k/x_n)$.

 (c) We prove the statement  by induction on $\ell$.
For $\ell=1$, the assertion from  (b).   Now let $\ell\geq 2$. By induction  hypothesis we may assume that   $J_{\ell-1}(I^k)=B^{k\cb-(\ell-1)\eb}(u^k/x_n^{\ell-1})$. Then   $J_{\ell-1}(I^k)$ is $(k\cb-(\ell-1)\eb)$-bounded  strongly stable. By Theorem~\ref{d-1}  it follows that  $J_{\ell}(I^k)$ is  $(k\cb-\ell\eb)$-bounded  generated in degree $kd-\ell$  and
\begin{eqnarray}
\label{weneedit} J_{\ell}(I^k)= \{w/x_n\:\; w\in G(J_{\ell-1}(I^k)),\;  x_n|w \text{ and }  \Deg (w/x_n)\leq k\cb-\ell\eb\}.
\end{eqnarray}
Now we prove that $J_{\ell}(I^k)=B^{k\cb-\ell\eb}(u^k/x_n^\ell)$.

Let $v\in B^{k\cb-\ell\eb}(u^k/x_n^\ell)$, then $v\prec u^k/x_n^{\ell}$  and $\Deg (v)\leq k\cb-\ell\eb$. This implies that $x_nv\prec u^k/x_n^{\ell-1}$, and hence  $vx_n\in B^{k\cb-(\ell-1)\eb}(u^k/x_n^{\ell-1})$.  By induction hypothesis, $B^{k\cb-(\ell-1)\eb}(u^k/x_n^{\ell-1})=J_{\ell-1}(I^k)$, and  so  $x_nv\in G(J_{\ell-1}(I^k))$. Hence (\ref{weneedit}) implies that $v\in J_\ell(I^k)$

Conversely, let  $v\in J_{\ell}(I^k)$.  Then by (\ref{weneedit}), $v=w/x_n$, with   $x_n|w$, $w\in G(J_{\ell-1}(I^k))$ and $\Deg (w/x_n)\leq k\cb-\ell\eb$.
It follows that $\Deg (w)\leq k\cb-(\ell-1)\eb$.
Since $J_{\ell-1}(I^k)=B^{k\cb-(\ell-1)\eb}(u^k/x_n^{\ell-1})$ by induction hypothesis, we have $w\prec u^k/x_n^{\ell-1}$ and $w$ is $k\cb-(\ell-1)\eb$-bounded. Since $x_n|w$, we get $x_n|(u^k/x_n^{\ell-1})$ and $w/x_n\prec u^k/x_n^{\ell}$. It follows that  $w/x_n\prec u^k/x_n^{\ell}$ and is $(k\cb-\ell\eb)$-bounded. Hence $J_{\ell}(I^k)\subseteq B^{k\cb-\ell\eb}(u^k/x_n^\ell)$.
\end{proof}

\begin{Remarks}
{\em
(a) The product of two $\cb$-bounded strongly stable  ideals is not necessarily a  $\cb$-bounded strongly stable  ideal.

 For example, let $I=(x_1x_2,x_1x_3,x_1x_4,x_2x_3)\subset K[x_1,x_2,x_3,x_4]$.  Then
\[
I^2=(x_1^2x_2^2,x_1^2x_2x_3,x_1^2x_2x_4,x_1x_2^2x_3,x_1^2x_3^2,
 x_1^2x_3x_4,x_1x_2x_3^2,x_1^2x_4^2,x_1x_2x_3x_4,x_2^2x_3^2).
 \]
The ideal  $I$ is $(1,1,1,1)$-bounded strongly stable. Since $x_1x_2^2x_4=x_2\frac{(x_1x_4)(x_2x_3)}{x_3}\notin I^2$, we see that
 $I^2$ is not $2\cb$-bounded strongly stable. Therefore,  Theorem~\ref{d-1} cannot be used to compute $\sat(I^2)$.

(b) A statement similar  to Theorem~\ref{soclepower} (a) does not hold for $\cb$-bounded stable principal ideals.

For example, let
 $u=x_1x_2x_3\in K[x_1,x_2,x_3]$ and  $\cb=(2,2,2)$. Then $\B^\cb(u)=(x_1^2x_2,x_1x_2^2,x_1x_2x_3)$, and
 \[
 (\B^\cb(u))^2=(x_1^4x_2^2,x_1^3x_2^3,x_1^2x_2^4,x_1^3x_2^2x_3,x_1^2x_2^3x_3,x_1^2x_2^2x_3^2).
 \]
On the other hand,
\[
\B^{2\cb}(u^2)=(x_1^4x_2^2,x_1^3x_2^3,x_1^2x_2^4,x_1^3x_2^2x_3,x_1^2x_2^3x_3,x_1^2x_2^2x_3^2,x_1^4x_3^2,x_1^3x_2x_3^2,x_1^4x_2x_3).
\]

(c)  In general, $B^{\cb}(u_1)B^{\cb}(u_2)\neq B^{2\cb}(u_1u_2)$. Indeed, let $u_1=x_1x_2^2$, $u_2=x_1x_3^2$ and $\cb=(2,2,2)$. Then \[B^{\cb}(u_1)B^{\cb}(u_2)=(x_1^4x_2^2,x_1^3x_2^3,x_1^2x_2^4,\\
x_1^4x_2x_3,x_1^3x_2^2x_3,x_1^2x_2^3x_3,x_1^3x_2x_3^2,x_1^2x_2^2x_3^2)\]
and
\[B^{2\cb}(u_1u_2)=(x_1^4x_3^2, x_1^4x_2^2,x_1^3x_2^3,x_1^2x_2^4,x_1^4x_2x_3,\\
x_1^3x_2^2x_3,x_1^2x_2^3x_3,x_1^3x_2x_3^2,x_1^2x_2^2x_3^2).\] }
\end{Remarks}

\medskip
For the powers of $\cb$-bounded strongly stable principal ideals, we have
\begin{Corollary}
\label{borelpower}
Let $u=x_1^{a_1}\cdots x_n^{a_n}$ be a $\cb$-bounded monomial in $S$ and $I=B^{\cb}(u)$.  Then for any positive integer $k$
\[
\sat(I^k)=\max\{\ell:\; \text{there exists}\ v\in G(B^{k\cb}(u^k))\ \text{with}\  x_n^\ell|v\ \text{and}\ \Deg(v/x_n^\ell)\leq k\cb-\ell\eb\}.
\]
\end{Corollary}
\begin{proof} From Theorem~\ref{soclepower}(a), we know  $I^k=B^{k\cb}(u^k)$. It follows that $I^k$ is $k\cb$-bounded strongly stable, the desired statement from  Theorem~\ref{sat}.
\end{proof}

A special case of $\cb$-bounded strongly stable principal ideals are the so-called Veronese type ideals, as shown in \cite{HMRZ}. For this class of ideals we have a more precise information about the saturation number. This will be discussed in the next section.

\section{The saturation number of powers of Veronese type ideals}

In this section  we consider a special class of  $\cb$-bounded strongly stable ideals, that is,
 Veronese type ideals.  Given a positive integers  $n$, and an integer $d$ and   an  integer vector $\ab=(a_1,\ldots,a_n)$ with
 $a_1\geq a_2\geq \cdots\geq a_n$, one defines the
monomial ideal  $I_{\ab,n,d}\subset S=K[x_1,\ldots,x_n]$ with
\[
G(I_{\ab,n,d})=\{x_1^{b_1}x_2^{b_2}\cdots x_n^{b_n}\; \mid \; \sum_{i=1}^nb_i=d \text{ and  $b_i\leq a_i$ for $i=1,\ldots,n$}\}.
\]
It is obvious that $I_{\ab,n,d}$ is $\cb$-bounded strongly stable.

For the proof of the next result we need the following simple  result.

\begin{Lemma}
\label{veronsedepth}
The following conditions are equivalent:
\begin{enumerate}
\item[(a)] $I_{\ab,n,d}=0$.
\item[(b)] {\em (i)} $a_i<0$ for some $i$, or {\em(ii)} $\sum_{i=1}^na_i<d$, or {\em(iii)} $d<0$.
\end{enumerate}
\end{Lemma}

\begin{proof} (b) $\Rightarrow$ (a)  is obvious.

(a) $\Rightarrow$ (b) Assume that $d, a_i\geq 0$ for all $i$, and  $\sum_{i=1}^na_i\geq d$.
Let $t$ be the smallest integer such that $\sum_{i=1}^ta_t\geq d$. Then $x_1^{a_1}\cdots x_{t-1}^{a_{t-1}}x_t^r\in I_{\ab,n,d}$, where $r=d-\sum_{i=1}^{t-1}a_i\leq a_t$, a contradiction.
\end{proof}

\medskip
In the following theorem we give a formula for $\sat(I_{\ab,n,d})$.  We assume that $\sum_{i=1}^na_i\geq d$  and $a_n\geq 0$, because otherwise $I_{\ab,n,d}=0$. We also assume that $n>1$. Because if $n=1$, then $\sat((x_1^d))=d$, and nothing is to prove.

\begin{Lemma}
\label{colonveronese}
For any Veronese ideal $I_{\cb,n,g}$  with  $\cb=(c_1,\ldots,c_n)$ and
 $c_1\geq c_2\geq \cdots\geq c_n$,  we have
\[
I_{\cb,n,g}:\mm=I_{\cb,n,g}+I_{\cb-\eb,n, g-1},
\]
where $\eb=(1,1,\ldots,1)$. In particular, $J_1(I_{\cb,n,g})=I_{\cb-\eb,n, g-1}$.
\end{Lemma}

\begin{proof} If $I_{\cb,n,g}=0$, then $I_{\cb-\eb,n, g-1}=0$ by Lemma~\ref{veronsedepth}. Assume now that $I_{\cb,n,g}\neq 0$. Then
 $g, c_i\geq 0$ for all $i$, and  $\sum_{i=1}^nc_i\geq g$, by Lemma~\ref{veronsedepth}.

  If $g=0$, then $I_{\cb-\eb,n, g-1}=0$ and $I_{\cb,n,g}=(1)$, and the assertion is trivial.

Now we assume that  $g\geq 1$. The inclusion $I_{\cb,n,g}+I_{\cb-\eb,n, g-1}\subseteq I_{\cb,n,g}:\mm$ is obvious. Conversely, let $v\in G(I_{\cb,n,g}:\mm)\setminus I_{\cb,n,g}$, Since  $I_{\cb,n,g}$ is $\cb$-bounded strongly stable, Theorem~\ref{d-1} implies that   $\deg(v)=g-1$  and $\Deg(v)\leq\cb-\eb$. Therefore,  $v\in I_{\cb-\eb,n,g-1}$.

Notice that  $I_{\cb,n,g}+J_{1}(I_{\cb,n,g})=I_{\cb,n,g}:\mm=I_{\cb,n,g}+I_{\cb-\eb,n,g-1}$. Since $I_{\cb,n,g}$ has a $g$-linear resolution, we get $J_1(I_{\cb,n,g})$ is generated in degree $g-1$. It follows that
$J_1(I_{\cb,n,g})=I_{\cb-\eb,n, g-1}$.
\end{proof}

\begin{Theorem}
\label{satveronese}
Let $I_{\ab,n,d}$ be a  Veronese type ideal with $n>1$, $d\geq 0$,  $a_1\geq a_2\geq \cdots \geq a_n\geq 0$ and $\sum_{i=1}^na_i\geq d$. Then
\begin{enumerate}
\item[(a)]  for all $\ell\geq 0$,  $J_{\ell}(I_{\ab,n,d})=I_{\ab-\ell\eb,n,d-\ell}$;
\item[(b)] $\sat(I_{\ab,n,d})=\min\bigg\{\bigg\lfloor \frac{\sum\limits_{i=1}^{n}a_i-d}{n-1}\bigg\rfloor, a_n, d\bigg\}$.
\end{enumerate}
where $\lfloor  a\rfloor$ is the largest integer less than or equal  to $a$.
\end{Theorem}

\begin{proof}
We prove (a) by induction on $\ell$. For $\ell=0$, the assertion is trivial.

Next let $\ell>1$. By induction hypothesis,  $J_{i}(I_{\ab,n,d})=I_{\ab-i\eb,n,d-i}$ for $i=0,\ldots, \ell-1$. Since each $I_{\ab-i\eb,n,d-i}$ has $(d-i)$-linear resolution, we may apply Lemma~\ref{jk}(d),  and together with Lemma~\ref{colonveronese} we obtain
\[
J_{\ell}(I)=J_1(J_{\ell-1}(I))=J_1(I_{\ab-(\ell-1)\eb)\eb,n,d-(\ell-1)})=I_{\ab-\ell\eb,n,d-\ell}.
\]

(b) By Lemma~\ref{jk}(b), we know $\sat(I_{\ab,n,d})=\max\{\ell\:\; J_\ell(I_{\ab,n,d})\neq 0\}$. It follows  from (a) and Lemma~\ref{veronsedepth}
\begin{eqnarray*}
\sat(I_{\ab,n,d})&=&\max\{\ell\:\; I_{\ab-\ell\eb,n,d-\ell}\neq 0\}\\
&=&\max\{\ell\:\; a_n-\ell\geq 0\ \text{and}\ d-\ell\geq 0\ \text{and}\  \sum_{i=1}^n(a_i-\ell)\geq d-\ell \}\\
&=&\max\{\ell\:\; \ell\leq a_n\ \text{and}\  \ell\leq d\ \text{and}\   \ell\leq \frac{\sum_{i=1}^na_i-d}{n-1} \}\\
&=&\min\bigg\{\bigg\lfloor \frac{\sum\limits_{i=1}^{n}a_i-d}{n-1}\bigg\rfloor, a_n, d\bigg\}.
\end{eqnarray*}
\end{proof}

\begin{Corollary}
\label{satveronese}
Let $I_{\ab,n,d}$ be a  Veronese type ideal with $n>1$, $d\geq 0$,  $a_1\geq a_2\geq \cdots \geq a_n\geq 0$ and $\sum_{i=1}^na_i\geq d$. Then for any $k$
\[
\sat((I_{\ab,n,d})^k)=\min\Bigg\{\bigg\lfloor \frac{(\sum\limits_{i=1}^{n}a_i-d)k}{n-1} \bigg\rfloor,ka_n,kd\Bigg\}.
\]
\end{Corollary}
\begin{proof} By \cite[Lemma 5.1]{HRV}, we obtain that $(I_{\ab,n,d})^k=I_{k\ab,n,kd}$, the desired statements follow from
Theorem \ref{satveronese}.
\end{proof}

\medskip
 \begin{Remark}{\em
 The product of two $\cb$-bounded Veronese type ideals is not necessarily a  $\cb$-bounded Veronese type ideal.

 For example, let $\ab=(3,3,1,2)$, $\bb=(2,2,0,1)$,  $\cb=\ab+\bb=(5,5,1,3)$, $d_1=6$,  $d_2=5$, $d_3=d_1+d_2=11$ and  $n=4$.
Then
\begin{eqnarray*}
I_{\ab,n,d_1}\!&\!=\!&\!(x_1^3x_2^3,x_1^3x_2^2x_3,x_1^2x_2^3x_3,x_1^3x_2^2x_4,
x_1^2x_2^3x_4,x_1^3x_2x_3x_4,x_1^2x_2^2x_3x_4,x_1x_2^3x_3x_4,
x_1^3x_2x_4^2,\\
& &x_1^2x_2^2x_4^2,x_1x_2^3x_4^2,x_1^3x_3x_4^2,x_1^2x_2x_3x_4^2,x_1x_2^2x_3x_4^2,x_2^3x_3x_4^2),\\
I_{\bb,n,d_2}&=&(x_1^2x_2^2x_4).
\end{eqnarray*}
It follows that
\begin{eqnarray*}
I_{\ab,n,d_1}\cdot I_{\bb,n,d_2}&=&(x_1^5x_2^5x_4,x_1^5x_2^4x_3x_4,x_1^4x_2^5x_3x_4,
x_1^5x_2^3x_3x_4^2,x_1^4x_2^4x_3x_4^2,x_1^3x_2^5x_3x_4^2,x_1^5x_2^3x_4^3,\\
& &x_1^4x_2^4x_4^3,x_1^3x_2^5x_4^3,x_1^5x_2^2x_3x_4^3,
x_1^4x_2^3x_3x_4^3,x_1^3x_2^4x_3x_4^3,x_1^2x_2^5x_3x_4^3,x_1^3x_2^5x_3x_4^2,\\
& &x_1^5x_2^4x_4^2,x_1^4x_2^5x_4^2).
\end{eqnarray*}
 However
 \begin{eqnarray*}
I_{\cb,n,d_3}&=&(x_1^5x_2^5x_3,x_1^5x_2^5x_4,x_1^5x_2^4x_3x_4,x_1^4x_2^5x_3x_4,
x_1^5x_2^3x_3x_4^2,x_1^4x_2^4x_3x_4^2,x_1^3x_2^5x_3x_4^2,x_1^5x_2^3x_4^3,\\
& &x_1^4x_2^4x_4^3,x_1^3x_2^5x_4^3,x_1^5x_2^2x_3x_4^3,
x_1^4x_2^3x_3x_4^3,x_1^3x_2^4x_3x_4^3,x_1^2x_2^5x_3x_4^3,x_1^3x_2^5x_3x_4^2,x_1^5x_2^4x_4^2,\\
& &x_1^4x_2^5x_4^2).
\end{eqnarray*}
  }
\end{Remark}

\medskip

A function $f: \mathbb{Q}\to \mathbb{Q}$ is called {\em quasi-linear}, if there exists an integer $m\geq 1$ and for each $i=0, \ldots, m-1$, a linear
function $f_i(x)=p_ix+q_i$ with $p_i, q_i\in \mathbb{Q}$ such that $f(k)=f_i(k)$ for $k\equiv i\ \mod m$.

\medskip
For  Veronese type ideals, we can give concrete quasi-linear functions describing the saturation number of the powers.
\begin{Theorem}
\label{qusilear}
Let $I_{\ab,n,d}$ be a    Veronese type ideal  with $n>1$, $d\geq 0$,  $a_1\geq a_2\geq \cdots \geq a_n\geq 0$ and $\sum_{i=1}^na_i\geq d$.
Let $\sum\limits_{i=1}^{n}a_i-d=s(n-1)$ with $s\in\QQ$ and
$t=\min\{ s, a_n, d\}$.
\begin{enumerate}
\item[(a)]If $t=s$, then  $\sat((I_{\ab,n,d})^k)=p_ik+q_i$
where $p_i=s$, $q_i=\lfloor si\rfloor-si$.
\item[(b)]If $t=a_n$, then $\sat((I_{\ab,n,d})^k)=a_nk$.
\item[(c)] If $t=d$, then $\sat((I_{\ab,n,d})^k)=dk$.
\end{enumerate}
\end{Theorem}
\begin{proof}
(a) If $t=s$, then $s\leq \min\{a_n,d\}$. Thus $ks\leq \min\{ka_n,kd\}$.
It follows that $\lfloor ks\rfloor \leq \min\{ka_n,kd\}$.
By Corollary \ref{satveronese},  we obtain $$\sat((I_{\ab,n,d})^k)=\lfloor ks\rfloor.$$
Let $k\equiv i\ \mod(n-1)$, then $k=(n-1)\ell+i$ with $0\leq i <n-1$.
It follows that
\[
 ks=s(n-1)\ell+si=(\sum\limits_{i=1}^{n}a_i-d)\ell+si,
\]
Hence
\begin{eqnarray*}
\lfloor ks\rfloor&=&(\sum\limits_{i=1}^{n}a_i-d)\ell+\lfloor si\rfloor=(\sum\limits_{i=1}^{n}a_i-d)\frac{k-i}{n-1}+\lfloor si\rfloor\\
&=&\frac{\sum\limits_{i=1}^{n}a_i-d}{n-1}(k-i)+\lfloor si\rfloor=s(k-i)+\lfloor si\rfloor\\
&=&sk+\lfloor si\rfloor-si.
\end{eqnarray*}
Choose $p_i=s$, $q_i=\lfloor si\rfloor-si$, we have $\sat((I_{\ab,n,d})^k)=p_ik+q_i$.

(b) If $t=a_n$, then $a_n\leq \min\{s,d\}$. It follows that  $ka_n\leq \min\{ks,kd\}$.
By Corollary \ref{satveronese}, $\sat((I_{\ab,n,d})^k)=a_nk$.

(c) If $t=d$, then $d\leq \min\{s,a_n\}$. It follows that  $kd\leq \min\{ks,ka_n\}$.
By Corollary~\ref{satveronese}, $\sat((I_{\ab,n,d})^k)=dk$.
\end{proof}

\end{document}